\documentclass[12pt, a4paper]{article}
\usepackage{latexsym,amssymb,amsfonts,amsmath,epsfig,tabularx}
\usepackage{amsthm}

\usepackage{hyperref}
\usepackage{graphics}
\usepackage{verbatim}
\usepackage{enumitem} 
\usepackage[usenames]{xcolor}

\newtheorem{theorem}{Theorem}
\newtheorem{remark}[theorem]{Remark}
\newtheorem{lemma}[theorem]{Lemma}
\newtheorem{proposition}[theorem]{Proposition}
\newtheorem{corollary}[theorem]{Corollary}

\newtheorem{example}[theorem]{Example}

\def\rho{\varrho}

\newcommand{\abs}[1]{\left\vert#1\right\vert}
\newcommand{\norm}[1]{\left\Vert#1\right\Vert}

\newcommand{\setu}{{\mathfrak{u}}}
\newcommand{\setv}{{\mathfrak{v}}}

\newcommand{\setU}{{\mathcal{U}}}

\newcommand{\QMC}{{\mathrm{QMC}}}

\newcommand{\rd}{{\mathrm{\,d}}}

\newcommand{\calF}{{\mathcal{F}}}

\newcommand{\calP}{{\mathcal{P}}}

\newcommand{\bsx}{{\boldsymbol{x}}}

\newcommand{\bst}{{\boldsymbol{t}}}

\newcommand{\bssigma}{{\boldsymbol{\sigma}}}
\newcommand{\bszero}{{\boldsymbol{0}}}
\newcommand{\bsone}{{\boldsymbol{1}}}

\newcommand{\bsgamma}{{\boldsymbol{\gamma}}}

\newcommand{\bspitch}{{\boldsymbol{\,\pitchfork}}}

\definecolor{darkred}{RGB}{139,0,0}
\definecolor{darkgreen}{RGB}{0,100,0}
\definecolor{darkmagenta}{RGB}{139,0,139}
\definecolor{darkorange}{RGB}{190,70,20}

\def\cbdelete[#1]{}

\newcommand{\bbR}{{\mathbb{R}}}

\newcommand{\NN}{{\mathbb{N}}}

\newcommand{\mask}[1]{}

\title{On Quasi-Monte Carlo Methods\\ in Weighted ANOVA Spaces}

\author{P. Kritzer\thanks{P. Kritzer is supported by the Austrian
Science Fund (FWF): Project F5506-N26, which is a part of the Special Research Program
"Quasi-Monte Carlo Methods: Theory and Applications".}, F. Pillichshammer\thanks{F. Pillichshammer is supported by the Austrian Science Fund (FWF): 
Project F5509-N26, which is a part of the Special Research Program 
"Quasi-Monte Carlo Methods: Theory and Applications".}, and G. W. Wasilkowski}

\date{\today}

\begin{document}
\maketitle

\begin{abstract}
\noindent In the present paper we study quasi-Monte Carlo
rules for approximating integrals over
the $d$-dimensional unit cube for functions from weighted Sobolev spaces
of regularity one. While the properties of these rules are well understood
for anchored Sobolev spaces, this is not 
the case for the ANOVA
spaces, which are another very important type of reference spaces
for quasi-Monte Carlo rules.

Using a direct approach we provide a formula for the worst case error
of quasi-Monte Carlo rules for functions from weighted ANOVA spaces.
As a consequence we 
bound the worst case error from above in terms of weighted discrepancy of the employed integration nodes. 
On the other hand we also obtain a general lower bound in terms of the number $n$ of used integration nodes.

For the one-dimensional case our results lead to the optimal
integration rule 
and also in the two-dimensional case we 
provide rules yielding optimal convergence rates.
\end{abstract}

\centerline{\begin{minipage}[hc]{130mm}{
}
\end{minipage}}

\section{Introduction}

In this paper we study numerical integration of functions $f$ over the $d$-dimensional unit cube $[0,1]^d$. 
A powerful method is the quasi-Monte Carlo (QMC) method which approximates the integral by 
$$
\int_{[0,1]^d} f(\bsx) \rd \bsx \approx \frac{1}{n}\sum_{j=1}^n f(\bst_j),
$$ 
where the $\bst_j$ are given points in $[0,1]^d$. The latter expression is called an ($n$-point) QMC rule and is denoted by ${\rm QMC}_{d,n}$, 
with $d$ indicating the dimension. We call the set $\calP=\{\bst_1,\ldots,\bst_n\}$ the underlying node set of the QMC rule.  General introductions to QMC methods can, e.g., be found in \cite{DP2010,LP2014,niesiam}.

The basis of the QMC method is the fact that the absolute error of a
QMC rule
\[
\left|\int_{[0,1]^d} f(\bsx) \rd \bsx - {\rm QMC}_{d,n}(f)\right|
\]
can be separated into properties of the integrand $f$ on the one hand, and distribution properties of the node set underlying the QMC rule on the other hand. 
Such estimates are called Koksma-Hlawka type inequalities, which go back to Koksma~\cite{K1942} (for $d=1$) and Hlawka~\cite{H1961} (for arbitrary $d \in \mathbb{N}$). 
In classical cases the absolute error is bounded by the product of the variation of the integrand $f$ in the sense of Hardy and Krause, and the star-discrepancy of the node set $\calP$. 
Until today, several variants of these classical results for various function classes have been developed, see, e.g., \cite{AD2015,BCGT2013,H1998,NW2010,PS2015,SW1998} to mention just a few 
references.

Nowadays it is very convenient to introduce  Koksma-Hlawka type
inequalities as equalities for the worst case integration errors
of QMC rules for functions from Banach spaces
$(\calF,\|\cdot\|_{\calF})$ which are defined as
\[
{\rm error}({\rm QMC}_{d,n};\calF)=
\sup_{f \in \calF \atop \|f\|_{\calF} \le 1}
\left|\int_{[0,1]^d} f(\bsx) \rd \bsx - {\rm QMC}_{d,n}(f)\right|.
\]
For example, the classical Koksma-Hlawka inequality can then be stated
as
\[{\rm error}({\rm QMC}_{d,n};\calF)=L_{p^\ast}(\calP),
\]
where $L_{p^\ast}(\calP)$ is the so-called $L_{p^\ast}$-discrepancy of the node set $\calP$ (see Section \ref{def:disc} below for a precise definition),
when considering  the norm 
\[
\|f\|_{\calF}=\left(\int_{[0,1]^d} \left|\frac{\partial^d f}{\partial \bsx}
(\bsx)\right|^p \rd \bsx\right)^{1/p} 
\]
for the space of functions anchored at $\bszero$,
see \cite[Sec.~3.1.5]{NW2008}, where $p, p^\ast \in [1,\infty]$ are 
conjugate, i.e., $1/p+1/p^\ast=1$. For further information we also
refer to \cite{DP2010,NW2010}.

The most prominent example of a function space in this context is the anchored Sobolev space of regularity one, and, even more general, the $\bsgamma$-weighted anchored Sobolev spaces of regularity one. The exact definitions of these spaces will be given in Section~\ref{def:anchspace}. For these Sobolev spaces the Koksma-Hlawka theory is very well understood: the worst case error is exactly the weighted $L_{p^\ast}$-discrepancy.

Related and also often considered as reference spaces for QMC rules
are the ($\bsgamma$-weighted) ANOVA spaces which will be formally
introduced in Section~\ref{def:ANOVAsp};
see, e.g., \cite{DKS2013,DP2010,HeRiWa2016,HKPW2019,NW2008,SW2002}. 
While for special choices of weights the anchored space and the ANOVA
space can be related in terms of embeddings
(see, e.g., \cite{GHHRW2017,HR2015,HeRiWa2016,HS2016,KPW2017})
and therefore the error
analysis for these two spaces is (up to embedding constants)
equivalent, this is not possible for general weights. 

In the present paper we provide a direct approach for the error analysis
of QMC rules in $\bsgamma$-weighted ANOVA spaces. It is an advantage that
this approach will work for general choices of non-negative weights
without any restriction (cf. Remark~\ref{re:altappr} in Section~\ref{sec:QMC}).

A further advantage appears when we restrict ourselves to the 1D case
(i.e., $d=1$).
Recall that the optimality of the composite midpoint rule for the $L_{p^*}$-discrepancy has been
  known for quite some time but only for $p^*\in\{2,\infty\}$, see
  \cite{nie1973}, and was very recently extended to arbitrary
  $p^*\in[1,\infty]$ in \cite{KPa19}. Using our results for the 1D case,
  we are  able
  to provide an elementary proof of the optimality of the composite
  midpoint rule for arbitrary $p^*$ 
for integration in the ANOVA space  
(cf. Theorem~\ref{lem:midp} in Section~\ref{sec:1D}). 

The paper is organized as follows. In Section~\ref{sec:basics} we
introduce the basic notation, the weighted ANOVA and anchored spaces
of regularity one, and several notions of discrepancy.
Section~\ref{sec:QMC} is devoted to the error analysis of QMC rules
in weighted ANOVA spaces. The main results show how the respective
worst case errors can be related to the weighted discrepancy of the
underlying node sets (cf. Corollary~\ref{cor1}).
Furthermore, we provide a general lower bound on the worst case error.
In Section~\ref{sec:lowD}  two subsections are
devoted to the 1D and the 2D cases, respectively. In the 1D case we
will show that the composite midpoint rule is the optimal QMC rule
among all QMC rules (Theorem~\ref{lem:midp}). In the 2D case we show
that, for example, shifted Hammersley point sets 
achieve the optimal
convergence rate of the error (Example~\ref{ex:Hampts}).

\section{Basic definitions and facts}\label{sec:basics}
We begin with the notation used in the paper.

\subsection{Notation}
For a positive integer $d$, we write $[d]$ to denote
\[[d]=\{1,2,\dots,d\}.
\]
We use $\setu$ and $\setv$ for monotonically increasing
sequences of numbers from $[d]$, e.g.,
\[\setu=(u_1,\dots,u_k), \quad\mbox{where}\quad
  1\le u_1 <\cdots< u_k\le d, \quad\mbox{and}\quad k=|\setu|.
\]
This includes the empty sequence $\setu=\emptyset$ with $|\emptyset|=0$.
Often, it is convenient to treat the $\setu$'s as sets, since then
we can write $\setu\subseteq[d]$, $j\in\setu$, $\setu\setminus\setv$, etc.

For $\bsx=(x_1,\dots ,x_d)\in[0,1]^d$ and $\setu\subseteq[d]$,
by $\bsx_\setu$ we mean the point in $[0,1]^{|\setu|}$
with the coordinates $x_j$ for $j\in\setu$. That is,
\[\bsx_\setu=(x_{u_1},\dots,x_{u_k})\quad\mbox{for}\quad
    \setu=(u_1,\dots,u_k).
\]

For $\setu\not=\emptyset$, we write $\partial^{(\setu)}$ to denote
mixed first order partial derivatives,
\[\partial^{(\setu)}=\prod_{j\in\setu}\frac\partial{\partial x_{u_j}}.
\]
For $\setu=\emptyset$, $\partial^{(\emptyset)}$ is the identity operator. 

We consider weights $\bsgamma=(\gamma_\setu)_{\setu\subseteq[d]}$, 
where the $\gamma_\setu$'s are nonnegative reals.
Sometimes we will use $\setU_+$ to list the sequences corresponding to positive weights,
\[\setU_+=\{\setu\subseteq[d]\,:\,\gamma_\setu>0\}.
\]

\subsection{$\bsgamma$-weighted ANOVA spaces}\label{def:ANOVAsp}

For given $d$, weights $\bsgamma$, and $p\in[1,\infty]$, the corresponding
space $\calF_d=\calF_{d,p,\bsgamma}$ is the Banach space of functions
\[f:[0,1]^d\to\bbR
\]
endowed with the norm
\[\|f\|_{\calF_d}=\bigg[\sum_{\setu\subseteq[d]}
  \gamma_\setu^{-p}\,\bigg\|\int_{[0,1]^{d-|\setu|}}
  \partial^{(\setu)}f(\cdot_\setu;\bst_{[d]\setminus\setu})
  \rd\bst_{[d]\setminus\setu}\bigg\|_{L_p([0,1]^{|\setu|})}^p
  \bigg]^{1/p}\quad\mbox{if}\quad p<\infty,
\]
and
\[\|f\|_{\calF_d}=\max_{\setu\subseteq[d]}
  \gamma_\setu^{-1}\,\bigg\|\int_{[0,1]^{d-|\setu|}}
  \partial^{(\setu)}f(\cdot_\setu;\bst_{[d]\setminus\setu})
  \rd\bst_{[d]\setminus\setu}\bigg\|_{L_\infty([0,1]^{|\setu|})}
  \quad\mbox{if}\quad p=\infty.
\]
By a convention $0/0=0$ so that for $\gamma_\setu=0$ the
corresponding integral part of the definition is also zero.
Note also that for $\setu=\emptyset$, 
\[\int_{[0,1]^{d-|\emptyset|}}
  \partial^{(\emptyset)}f(\cdot_\emptyset;\bst_{[d]\setminus\emptyset})
  \rd\bst_{[d]\setminus\emptyset} = \int_{[0,1]^d}f(\bst)\rd\bst,
\]
and for $\setu=[d]$ the above $L_p$-norms equal 
$$
\left\|\frac{\partial^d}{\partial x_1 \cdots \partial x_d} f\right\|_{L_p([0,1]^d)}\ \ \ \mbox{ for all $p \in [1,\infty]$.}
$$
  
Consider next the ANOVA decomposition of functions $f\in\calF_d$,
\[f=\sum_{\setu\subseteq[d]}f_\setu,
\]
where
\[f_\emptyset=\int_{[0,1]^d}f(\bsx)\rd\bsx
\]
and, for nonempty $\setu$, $f_\setu$ depends only on $\bsx_\setu$, and
\[\int_0^1f_\setu(\bst)\rd t_j=0\quad\mbox{for any\ }j\in\setu.
\]

From \cite{HeRiWa2016} we know that $f_\setu\equiv0$ for
$\setu\notin\setU_+$, i.e.,
\[f\,=\,\sum_{\setu\in\setU_+}f_\setu.
\]
Moreover, for nonempty $\setu\in\setU_+$
\[
  f_\setu(\bsx)=\int_{[0,1]^{|\setu|}}h_\setu(\bst_\setu)\,K_\setu(\bsx_\setu,
  \bst_\setu) \rd\bst_\setu,
\]
where $h_\setu\in L_p([0,1]^{|\setu|})$ and
\[K_\setu(\bsx_\setu,\bst_\setu)=\prod_{\ell\in\setu}K(x_\ell,t_\ell)
\quad\mbox{with}\quad K(x,t)  
=\left\{\begin{array}{ll} t &\mbox{if\ }x\ge t,\\
    t-1&\mbox{if\ }x<t.\end{array}\right.
\]
More precisely, let $F_\emptyset$ be the space of constant functions
with the absolute value as its norm. For nonempty
$\setu\in\setU_+$, let
\[F_\setu=\left\{f_\setu=\int_{[0,1]^{|\setu|}}h_\setu(\bst_\setu)\,
    K_\setu(\cdot,\bst_\setu)\rd\bst_\setu\,:\,
    h_\setu\in L_p([0,1]^{|\setu|})\right\},
\]
which is a Banach space with the norm
\[\|f_\setu\|_{F_\setu}=\|h_\setu\|_{L_p([0,1]^{|\setu|})}.
\]
Then
\[\calF_d=\bigoplus_{\setu\in\setU_+}F_\setu\quad\mbox{and}
\quad \|f\|_{\calF_d}=\bigg[\sum_{\setu\in\setU_+}
  \gamma_\setu^{-p}\,\|f_\setu\|_{F_\setu}^p\bigg]^{1/p}
\]
with the obvious modifications if $p=\infty$.

\subsection{$\bsgamma$-weighted anchored spaces}\label{def:anchspace}

Consider the following Banach space
$\calF^{\bspitch}_d=\calF^{\bspitch}_{d,p,\bsgamma}$ of functions
\[  f:[0,1]^d\to\bbR
\]
whose norm is given by
\[\|f\|_{\calF^{\bspitch}_d}=\bigg[\sum_{\setu\in\setU_+}
  \gamma_\setu^{-p}\,\left\|
  \partial^{(\setu)}f(\cdot_\setu;\bszero_{[d]\setminus\setu})
  \right\|_{L_p([0,1]^{|\setu|})}^p
  \bigg]^{1/p}\quad\mbox{if}\quad p<\infty
\]
and
\[\|f\|_{\calF^{\bspitch}_d}=\max_{\setu\in\setU_+}
  \gamma_\setu^{-1}\,\left\|
  \partial^{(\setu)}f(\cdot_\setu;\bszero_{[d]\setminus\setu})
  \right\|_{L_\infty([0,1]^{|\setu|})}
  \quad\mbox{if}\quad p=\infty.
\]

Consider next the {\em anchored} decompositions of
$f\in\calF_d^{\bspitch}$,
\[ f=\sum_{\setu\in\setU_+}f_\setu^{\bspitch},
\]
where, for $\setu\not=\emptyset$, $f_\setu^{\bspitch}$ depends only on
$\bsx_\setu$ and
\[f_\setu^{\bspitch}(\bsx_\setu)=0 \quad\mbox{if $x_j=0$ for some $j\in\setu$}.
\]

We know from \cite{HeRiWa2016} that for nonempty $\setu\in\setU_+$
\[ f_\setu^{\bspitch}(\bsx)=\int_{[0,1]^d}h_\setu(\bst_\setu)\,
    K_\setu^{\bspitch}(\bsx_\setu,\bst_\setu)\rd\bst_\setu,
\]
where $h_\setu\in L_p([0,1]^{|\setu|})$ and
\[K^{\bspitch}_\setu(\bsx_\setu,\bst_\setu)\,=\,
\prod_{\ell\in\setu}K^{\bspitch}(x_\ell,t_\ell)
\quad\mbox{with}\quad K^{\bspitch}(x,t)
 \, =\,\left\{\begin{array}{ll} 1&\mbox{if\ }x\ge t,\\
    0&\mbox{if\ }x<t.\end{array}\right.
\]  

As in the previous section, let $F_\emptyset^{\bspitch}=F_\emptyset$
be the space of constant functions with
the absolute value as its norm. For nonempty $\setu$, let
\[F^{\bspitch}_\setu=\left\{f^{\bspitch}_\setu=\int_{[0,1]^{|\setu|}}h_\setu(\bst_\setu)\,
    K^{\bspitch}_\setu(\cdot,\bst_\setu)\rd\bst_\setu\,:\,
    h_\setu\in L_p([0,1]^{|\setu|})\right\},
\]
which is a Banach space with the norm
\[\|f^{\bspitch}_\setu\|_{F^{\bspitch}_\setu}=
\|h_\setu\|_{L_p([0,1]^{|\setu|})}.
\]
Then
\[\calF^{\bspitch}_d=\calF^{\bspitch}_{d,p,\bsgamma}=
\bigoplus_{\setu\in\setU_+}F^{\bspitch}_\setu\quad\mbox{and}
\quad \|f\|_{\calF_d^{\bspitch}}=\bigg[\sum_{\setu\in\setU_+}
\gamma_\setu^{-p}\,\|f^{\bspitch}_\setu\|_{F^{\bspitch}_\setu}^p\bigg]^{1/p}
\]
is the $\bsgamma$-weighted anchored space of functions with anchor
$\bszero$.

We now recall the following relation between the ANOVA and the
 anchored spaces, see \cite[Proposition~13]{HeRiWa2016}.

\begin{proposition}
For any $p \in [1,\infty]$ and $\bsgamma$ the following holds.
The $\bsgamma$-weighted anchored and ANOVA spaces are equal
(as sets of functions) if and only if 
\begin{equation}\label{condweimset}
\gamma_{\setu}>0 \ \mbox{  implies  }\ \gamma_{\setv}>0 \ \mbox{ for all }\ \setv \subseteq \setu.
\end{equation}
Moreover, if \eqref{condweimset} does not hold, then
\[\calF^{\bspitch}_{d,p,\bsgamma}\not \subseteq \calF_{d,p,\bsgamma}\
\mbox{ and }\ \calF_{d,p,\bsgamma}\not \subseteq
\calF^{\bspitch}_{d,p,\bsgamma}.
\]
\end{proposition}

\subsection{Discrepancy and weighted discrepancy}\label{def:disc}

We now recall the definition of (weighted) discrepancy, which is
related to the errors of QMC methods studied in this
paper.

For a point set $\calP=\{\bsx_1,\ldots,\bsx_n\}$ in $[0,1)^d$ the
local discrepancy function $\Delta_{\calP}:[0,1]^d \rightarrow \bbR$
is defined as 
$$
\Delta_{\calP}(\bst)=\frac{|\{j \in \{1,\ldots,n\}\ : \ \bsx_j
  \in [\bszero,\bst)\}|}{n}-\lambda( [\bszero,\bst)),
$$ 
where $\bst=(t_1,\ldots,t_d)\in [0,1]^d$,
$[\bszero,\bst)=[0,t_1)\times \cdots \times [0,t_d)$ and
$\lambda( [\bszero,\bst))=t_1\cdots t_d$. 
The local discrepancy function can be expressed in terms of the
indicator function, namely, 
$$
\Delta_{\calP}(\bst)=\frac{1}{n} \sum_{j=1}^n
1_{[\bszero,\bst)}(\bsx_j)-\lambda( [\bszero,\bst)), 
$$ 
where $1_{[\bszero,\bst)}(\bsx_j)=1$ if $\bsx_j \in [\bszero,\bst)$
and 0 otherwise. Note that 
$$
1_{[\bszero,\bst)}(\bsx_j)=\prod_{i=1}^d 1_{[0,t_i)}(x_{j,i}), 
$$ 
where $x_{j,i}$ is the $i^{{\rm th}}$ component of $\bsx_j$.

For $p^\ast \in [1,\infty]$ the $L_{p^\ast}$-discrepancy of $\calP$
is defined as the $L_{p^\ast}$-norm of the local discrepancy function, i.e.,
\[L_{p^\ast}(\calP)=\|\Delta_{\calP}\|_{L_{p^*}([0,1]^d)}.\]

Furthermore, the $\bsgamma$-weighted $L_{p^\ast}$-discrepancy of $\calP$ is defined as
\[L_{p^\ast,\bsgamma}(\calP)= \left[\sum_{\emptyset\not=\setu\in\setU_+}
  \gamma_{\setu}^{p^\ast} \|\Delta_{\calP_{\setu}}\|_{L_{p^\ast}([0,1]^{|\setu|})}^{p^\ast} \right]^{1/p^\ast}\mbox{for}\quad p^*<\infty\]
and 
\[
L_{\infty,\bsgamma}(\calP)= \max_{\emptyset\not=\setu\in\setU_+}
\gamma_{\setu} \|\Delta_{\calP_{\setu}}\|_{L_{\infty}([0,1]^{|\setu|})}\quad \mbox{for}\quad p^*<\infty,
\]
where $\Delta_{\calP_{\setu}}$ denotes the local discrepancy function of the set that consists of the projected points of $\calP$ to the coordinates with indices in $\setu$. 
Weighted $L_{p^\ast}$-discrepancy was first introduced and studied by Sloan and Wo\'{z}niakowski~\cite{SW1998}. For further information 
on weighted discrepancy we also refer to \cite{DP2010,NW2010}.

\section{Quasi-Monte Carlo methods and their errors}\label{sec:QMC}
In this main section we consider QMC methods of the form
\[\QMC_{d,n}(f)=\frac{1}{n}\,\sum_{j=1}^nf(\bst_j)
\]
for some deterministically chosen points $\bst_j\in[0,1]^d$. We are interested in their worst
case errors with respect to the unit ball of the space
$\calF_d$ defined as 
\[{\rm error}(\QMC_{d,n};\calF_d)=\sup_{f \in \calF_d \atop \|f\|_{\calF_d} \le 1}
\left|\int_{[0,1]^d}f(\bsx)\rd\bsx-\QMC_{d,n}(f)\right|. 
\]

\begin{remark}\label{re:altappr}\rm
If Condition \eqref{condweimset} is satisfied, then 
to study
QMC in the ANOVA space one may consider the embedding operator
$\imath: \calF_d \rightarrow   \calF_d^{\bspitch}$, $\imath(f)=f$. Then
\[\frac1{\|\imath^{-1}\|} \,
{\rm error}(\QMC_{d,n};\calF_d^{\bspitch})\, \le\, 
     {\rm error}(\QMC_{d,n};\calF_d)\, \le\, \|\imath\| \,
     {\rm error}(\QMC_{d,n};\calF_d^{\bspitch}),
\]
where $\|\imath\|$ and $\|\imath^{-1}\|$ are the operator norms of
the embedding operator $\imath$ and its inverse $\imath^{-1}$,
respectively. It is well known (see, e.g., \cite{SW1998}) that 
$$
{\rm error}(\QMC_{d,n};\calF_d^{\bspitch})= L_{p^\ast,\bsgamma}(\overline{\calP}),
$$ 
where $L_{p^\ast,\bsgamma}(\overline{\calP})$ is the weighted
$L_{p^\ast}$-discrepancy of the point set\footnote{We remark that in
\cite{SW1998} the anchored space with anchor $\bsone$ is considered
which results in a worst case error of exactly
$L_{p^\ast,\bsgamma}(\calP)$, where $\calP$ is the node set of the
QMC rule. Here we have chosen the anchor as $\bszero$, and therefore
in the formula for the worst case error the point set
$\overline{\calP}$ appears.} 
$$
\overline{\calP}=\{\bsone - \bst_j \ : \ j=1,\ldots,n\},
$$ 
where $\bsone - \bst_j$ is defined as the component-wise difference of
the vector containing only ones and $\bst_j$.
Estimates or, in some cases, exact values of
$\max\{\|\imath\|,\|\imath^{-1}\|\}$ can be found in, e.g., 
\cite{GHHRW2017,HR2015,HeRiWa2016,HKPW2019,HS2016,KPW2016}
(see also Remark \ref{rem10}). 
Results on weighted discrepancy can be found in, e.g.,
\cite{DP2010,DP2015,LP2003,SW1998}.
\end{remark}

However, in order to follow the approach as sketched in
Remark~\ref{re:altappr} one requires the assumption that Condition
\eqref{condweimset} is satisfied. For example, this condition is not
satisfied for weights of the form $\gamma_{[d]}=1$ and
$\gamma_{\setu}=0$ for all $\setu \subsetneq [d]$. In the present
paper we follow a direct approach of an error analysis for QMC rules
in the $\bsgamma$-weighted ANOVA space that does not require
Condition \eqref{condweimset} and the embedding of the ANOVA space
into the anchored space. 

\subsection{A formula for the worst case error}

The following theorem gives a formula for
the worst case integration error.

\begin{theorem}\label{thm1}
For any QMC rule $\QMC_{d,n}$ 
\[{\rm error}(\QMC_{d,n};\calF_d) = \frac{1}{n}
  \bigg[\sum_{\emptyset\not=\setu\in\setU_+}\gamma_\setu^{p^*}\,
  \int_{[0,1]^{|\setu|}}\bigg|\sum_{j=1}^nK_\setu(\bsx_{j,\setu},\bst_\setu)
  \bigg|^{p^*}\rd\bst_\setu\bigg]^{1/p^*}\quad
  \mbox{for}\quad p^*<\infty
\]
and
\[{\rm error}(\QMC_{d,n};\calF_d) = \frac{1}{n}
  \max_{\emptyset\not=\setu\in\setU_+} \gamma_\setu\,\bigg\|\sum_{j=1}^n
  K_\setu(\bsx_{j,\setu},\cdot_\setu)\bigg\|_{L_\infty([0,1]^{|\setu|})}
  \quad\mbox{for}\quad p^*=\infty,
\]
where we write $\bsx_{j,\setu}$ short hand for $(\bsx_j)_{\setu}$.
\end{theorem}
\begin{proof}
We present the proof only for $p^*<\infty$ since it is very similar
for $p^*=\infty$.

In the following, let, for $\setu\subseteq [d]$, $\QMC_{\setu,n}$ denote the projection of the rule $\QMC_{d,n}$ onto 
those coordinates with indices in $\setu$ (i.e., the rule is based on $\abs{\setu}$-dimensional 
integration nodes obtained by projecting the nodes of $\QMC_{d,n}$ accordingly).
For any $f\in\calF_d$,
\begin{eqnarray}
  \bigg|\int_{[0,1]^d}f(\bsx)\rd\bsx-\QMC_{d,n}(f)\bigg|&=&
  \bigg|\sum_{\emptyset\not=\setu\in\setU_+}
  \QMC_{d,n}(f_\setu)\bigg|\nonumber\\
&=&\frac{1}{n}\,\bigg|\sum_{\emptyset\not=\setu\in\setU_+}\sum_{j=1}^n
  \int_{[0,1]^{|\setu|}}h_\setu(\bst_\setu)\,
  K_\setu(\bsx_{j,\setu},\bst_\setu)
    \rd\bst_\setu\bigg|\nonumber\\
&=& \frac{1}{n}\bigg|\sum_{\emptyset\not=\setu\in\setU_+}
 \int_{[0,1]^{|\setu|}}h_\setu(\bst_\setu)
 \sum_{j=1}^nK_\setu(\bsx_{j,\setu},\bst_\setu)\rd\bst_\setu\bigg|\nonumber\\
&\le&\frac{1}{n}\,\sum_{\emptyset\not=\setu\in\setU_+}\|h_\setu\|_{L_p}\,
 \bigg[\int_{[0,1]^{|\setu|}}\bigg|\sum_{j=1}^n
   K_\setu(\bsx_{j,\setu},\bst_\setu)
      \bigg|^{p^*}\rd\bst_\setu\bigg]^{1/p^*}\label{ineq:1}.
\end{eqnarray}
Using H\"older's inequality one more time we get another upper bound,
\begin{equation}\label{ineq:2}
\bigg|\int_{[0,1]^d}f(\bsx)\rd\bsx-\QMC_{d,n}(f)\bigg|\le \frac{1}{n}
\|f\|_{\calF_d}\,\bigg[\sum_{\emptyset\not=\setu\in\setU_+}
  \gamma_\setu^{p^*}\,
  \int_{[0,1]^{|\setu|}}\bigg|\sum_{j=1}^nK_\setu(\bsx_{j,\setu},\bst_\setu)
      \bigg|^{p^*}\rd\bst_\setu\bigg]^{1/p^*}.
\end{equation}

To prove equality, we will use the fact that H\"older's inequality
is sharp. For that purpose recall that for two functions $f$ and $g$ 
\[\int_D|g\,f|=\|g\|_{L_p(D)}\,\|f\|_{L_{p^*}(D)}\quad\mbox{if}\quad
  |g|=c|f|^{p^*-1}\mbox{\ a.e. on D},
\]
and for two sequences of numbers  $a_i$ and $b_i$
\[ \sum_i|a_i\,b_i|=\bigg[\sum_i|a_i|^p\bigg]^{1/p}\,
  \bigg[\sum_i|b_i|^{p^*}\bigg]^{1/p^*}\quad\mbox{if}\quad
  |a_i|=c|b_i|^{p^*-1}\mbox{\ for all $i$}. 
\]
Consider next the function $f$ with $h_\emptyset=0$ and 
\[  h_\setu(\bst_\setu)\,=\,
  \bigg|\sum_{j=1}^nK_\setu(\bsx_{j,\setu},\bst_\setu)\bigg|^{p^*-1}.
\]
For such $h_\setu$'s we have equality in \eqref{ineq:1}.
Moreover, for every $\setu\not=\emptyset$,
\begin{eqnarray*}
  \|h_\setu\|_{L_p([0,1]^{|\setu|})}&=&\bigg[\int_{[0,1]^{|\setu|}}
  \bigg|\sum_{j=1}^nK_\setu(\bsx_{j,\setu},\bst_\setu)\bigg|^{p(p^*-1)}\rd\bst_\setu
  \bigg]^{1/p}\\
&=&\bigg\|\sum_{j=1}^nK_\setu(\bsx_{j,\setu},\cdot_\setu)
\bigg\|_{L_{p^*}([0,1]^{|\setu|})}^{p^*-1},
\end{eqnarray*}
since $p(p^*-1)=p^*$ and $1/p=(p^*-1)/p^*$.
Therefore we have equality also in \eqref{ineq:2} which completes the proof.
\end{proof}

Next we relate the worst case error to the $L_{p^\ast}$-discrepancy of
the point sets underlying the QMC rule.

\begin{lemma}\label{le:sumprod}
For every nonempty $\setu \subseteq [d]$ 
we have
\[\sum_{j=1}^n K_\setu(\bsx_{j,\setu},\bst_\setu)  = n
\sum_{k=1}^{|\setu|} (-1)^k \sum_{\setv \subseteq \setu\atop
  |\setv|=k} \Delta_{\calP_{\setv}}(\bst_{\setv})
\prod_{i \in\setu\setminus\setv} t_i, 
\]
where $\Delta_{\calP_{\setv}}(\bst_{\setv})$ denotes the local
discrepancy function in $\bst_{\setv}$ of the set that consists of the points of $\calP$
projected onto the coordinates with indices in $\setv$.   
\end{lemma}

\begin{proof}
Observe that 
\begin{equation}\label{kersufo}
  \sum_{j=1}^n K_\setu(\bsx_{j,\setu},\bst_\setu) =
  \sum_{j=1}^n \prod_{i \in \setu}(t_i-1_{[0,t_i)}(x_{j,i})).
\end{equation}
Let us now rewrite the product above as
\begin{eqnarray*}
  \prod_{i \in \setu}(t_i-1_{[0,t_i)}(x_{j,i})) & = &
 \sum_{\setv\subseteq \setu} \prod_{i \in \setv}(-1_{[0,t_i)}(x_{j,i}))
 \prod_{i \in \setu\setminus\setv} t_i\\
& = & \sum_{k=0}^{|\setu|} \sum_{\setv \subseteq \setu\atop |\setv|=k}
  \prod_{i \in \setv} (-1_{[0,t_i)}(x_{j,i}))
  \prod_{i \in \setu\setminus\setv} t_i.
\end{eqnarray*}
We have 
\begin{eqnarray*}
\sum_{k=0}^{|\setu|}  \sum_{\setv \subseteq \setu\atop |\setv|=k} (-1)^{|\setv|}
  \prod_{i \in \setu} t_i  & = & \left(\prod_{i\in \setu} t_i\right)
\sum_{k=0}^{|\setu|} {|\setu| \choose k} (-1)^k = \left(\prod_{i\in \setu}t_i\right)
  (1+(-1))^{|\setu|}=0.
\end{eqnarray*}
Hence
\begin{eqnarray*}
\prod_{i \in \setu}(t_i-1_{[0,t_i)}(x_{j,i}))& = & \sum_{k=0}^{|\setu|}\sum_{\setv
 \subseteq\setu\atop|\setv|=k}\left[\prod_{i \in \setv} (-1_{[0,t_i)}(x_{j,i}))
 \prod_{i \in \setu\setminus\setv} t_i -(-1)^{|\setv|} \prod_{i\in\setu} t_i
 \right] \\
& = & \sum_{k=0}^{|\setu|} (-1)^k \sum_{\setv \subseteq \setu\atop |\setv|=k}
\left[\prod_{i \in \setv} 1_{[0,t_i)}(x_{j,i}) - \prod_{i \in \setv} t_i \right]
      \prod_{i \in \setu\setminus\setv} t_i.
\end{eqnarray*}
For $k=0$ we have $\setv=\emptyset$ and hence
\[\prod_{i \in \setv}1_{[0,t_i)}(x_{j,i}) - \prod_{i \in \setv} t_i=0.
\]
This means that it suffices to start the summation with the index 1. Hence
\begin{eqnarray*}
\prod_{i \in \setu}(t_i-1_{[0,t_i)}(x_{j,i})) =  \sum_{k=1}^{|\setu|}(-1)^k
\sum_{\setv \subseteq \setu\atop |\setv|=k}  \left[1_{[\bszero,\bst_{\setv})}
(\bsx_{j,\setv}) - \lambda([\bszero,\bst_{\setv})) \right]
\prod_{i \in \setu\setminus\setv} t_i.
\end{eqnarray*}
Summation over all $j=1,\ldots,n$ gives
$$
\sum_{j=1}^n\prod_{i \in \setu}(t_i-1_{[0,t_i)}(x_{j,i})) = 
\sum_{k=1}^{|\setu|} (-1)^k \sum_{\setv \subseteq \setu\atop  |\setv|=k}
n\,\Delta_{\calP_{\setv}}(\bst_{\setv})  \prod_{i \in \setu\setminus\setv} t_i.
$$
\end{proof}

The following corollary to Theorem~\ref{thm1} bounds the error
in terms of the weighted $L_{p^\ast}$-discrepancy of the node set
$\calP$ underlying the $\QMC$ rule for suitably modified weights
$\widetilde{\bsgamma}$.

\begin{corollary}\label{cor1}
For any QMC rule $\QMC_{d,n}$ and $p \in [1,\infty]$ we have 
\[{\rm error}(\QMC_{d,n};\calF_d) \le L_{p^\ast,\widetilde{\bsgamma}}(\calP),
\]
where the latter is the
$\widetilde{\bsgamma}$-weighted $L_{p^\ast}$-discrepancy of $\calP$
and 
$\widetilde{\bsgamma}=(\widetilde{\gamma}_{\setu})_{\setu\subseteq [d]}$
with 
$$
\widetilde{\gamma}_{\setu}=\left((p^*+1)^{|\setu|}
  \sum_{\setv\in\setU_+ \atop \setv \supseteq \setu} \gamma_\setv^{p^*}\,
 \left(\frac{2^{p^\ast -1}}{p^*+1}\right)^{|\setv|}\,
   \right)^{1/p^*} \quad\mbox{for}\quad p^*<\infty,
$$
and 
$$
\widetilde{\gamma}_{\setu}= 2^{|\setu|} \gamma_{\setu}
\quad\mbox{for}\quad p^*=\infty.
$$
\end{corollary}

For the proof of Corollary~\ref{cor1} we use the following
simple lemma.

\begin{lemma}\label{le:easy}
For $p^*\in  [1,\infty)$ and $x_k\ge 0$ for $k=1,2,\ldots,\ell$ we have
\[ \left(\sum_{k=1}^\ell x_k \right)^{p^\ast}\, \le\,
\ell^{p^\ast-1}  \sum_{k=1}^\ell x_k^{p^\ast}
\]
with equality if $x_1=\cdots=x_\ell$.
\end{lemma}

\begin{proof}
We use H\"older's inequality and the fact that $p^\ast/p=p^\ast -1$ to obtain
\begin{eqnarray*}
\left(\sum_{k=1}^\ell x_k \right)^{p^\ast} & = &
\left(\sum_{k=1}^\ell x_k \cdot 1 \right)^{p^\ast}\\
& \le & \left( \left(\sum_{k=1}^\ell x_k^{p^\ast}\right)^{1/p^\ast}  \left(\sum_{k=1}^\ell 1^p\right)^{1/p}\right)^{p^\ast}\\
& = & \ell^{p^\ast/p}  \sum_{k=1}^\ell x_k^{p^\ast}\\
& = & \ell^{p^\ast-1}  \sum_{k=1}^\ell x_k^{p^\ast}.
\end{eqnarray*}
\end{proof}

\begin{proof}[Proof of Corollary~\ref{cor1}]
Consider first $p>1$ and hence $p^* < \infty$. According
to Theorem~\ref{thm1} and Lemma~\ref{le:sumprod} we have 
\begin{eqnarray}\label{prcor1:eq1}
{\rm error}(\QMC_{d,n};\calF_d) & = & 
\left[\sum_{\emptyset\not=\setu\in\setU_+}\gamma_\setu^{p^*}\,
  \int_{[0,1]^{|\setu|}}\left| \sum_{k=1}^{|\setu|} (-1)^k
  \sum_{\setv \subseteq \setu\atop |\setv|=k}
  \Delta_{\calP_{\setv}}(\bst_{\setv}) \prod_{i \in\setu\setminus \setv}
  t_i     \right|^{p^*}\rd\bst_\setu\right]^{1/p^*}\label{eq:4} \\
& \le & \left[\sum_{\emptyset\not=\setu\in\setU_+}\gamma_\setu^{p^*}\,
  \int_{[0,1]^{|\setu|}}\left( \sum_{\emptyset \not=\setv \subseteq \setu} |\Delta_{\calP_{\setv}}(\bst_{\setv})|  \prod_{i \in\setu\setminus \setv} t_i
      \right)^{p^*}\rd\bst_\setu\right]^{1/p^*} \nonumber\\
& \le & \left[\sum_{\emptyset\not=\setu\in\setU_+}\gamma_\setu^{p^*}\,
 2^{|\setu| (p^\ast -1)} \int_{[0,1]^{|\setu|}} \sum_{\emptyset \not=\setv \subseteq \setu} |\Delta_{\calP_{\setv}}(\bst_{\setv})|^{p^*}   \prod_{i\in\setu\setminus \setv} t_i^{p^*}    \rd\bst_\setu\right]^{1/p^*},\nonumber 
\end{eqnarray}
where we applied Lemma~\ref{le:easy} to the innermost sum. Interchanging
the integral and the inner sum gives
\begin{eqnarray*}
\lefteqn{{\rm error}(\QMC_{d,n};\calF_d)}\\ & \le & 
\left[\sum_{\emptyset\not=\setu\in\setU_+}\gamma_\setu^{p^*}\,
 2^{|\setu| (p^\ast -1)}  \sum_{\emptyset \not=\setv \subseteq \setu} \left(\int_{[0,1]^{|\setv|}} |\Delta_{\calP_{\setv}}(\bst_{\setv})|^{p^*} \rd \bst_{\setv}\right)  \left( \int_0^1 t^{p^*}\rd t\right)^{|\setu|-|\setv|} \right]^{1/p^*}\\
& = &   \left[\sum_{\emptyset\not=\setu\in\setU_+}\gamma_\setu^{p^*}\,
 2^{|\setu| (p^\ast -1)}  \sum_{\emptyset \not=\setv \subseteq \setu} (L_{p^\ast}(\calP_{\setv}))^{p^*}  \left( \frac{1}{p^*+1}\right)^{|\setu|-|\setv|} \right]^{1/p^*}\\
& = &  \left[\sum_{\emptyset\not=\setu\in\setU_+}\gamma_\setu^{p^*}\,
 \left(\frac{2^{p^\ast -1}}{p^*+1}\right)^{|\setu|}\,
   \sum_{\emptyset \not=\setv \subseteq \setu} (p^*+1)^{|\setv|}\,
   (L_{p^\ast}(\calP_{\setv}))^{p^*} \right]^{1/p^*}.
\end{eqnarray*}
Now we interchange the order of summation and obtain in this way 
\begin{eqnarray*}
{\rm error}(\QMC_{d,n};\calF_d) & \le &
\left[\sum_{\emptyset\not=\setv\in\setU_+} \left( (p^*+1)^{|\setv|}
\sum_{\setu\in\setU_+ \atop \setu \supseteq \setv} \gamma_\setu^{p^*}\,
 \left(\frac{2^{p^\ast -1}}{p^*+1}\right)^{|\setu|}\,
  \right)\,
   (L_{p^\ast}(\calP_{\setv}))^{p^*} \right]^{1/p^*}\\
   & = & \left[\sum_{\emptyset\not=\setv\in\setU_+} \widetilde{\gamma}_{\setv}^{p^\ast} \,
   (L_{p^\ast}(\calP_{\setv}))^{p^*} \right]^{1/p^*} \\
   & = & L_{p^\ast, \widetilde{\bsgamma}}(\calP).
\end{eqnarray*}

If $p=1$, and hence $p^*= \infty$, we trivially have 
\[
{\rm error}(\QMC_{d,n};\calF_d) \le
\max_{\emptyset\not=\setu\in\setU_+}\gamma_\setu\,
\sum_{\emptyset\not=\setv\subseteq\setu} L_\infty(\calP_\setv),
\]
and from this we obtain 
\begin{eqnarray*}
{\rm error}(\QMC_{d,n};\calF_d) & \le & \max_{\emptyset\not=\setu\in\setU_+}
\gamma_\setu\,2^{|\setu|}\, L_\infty(\calP_\setu) =
L_{\infty,\widetilde{\bsgamma}}(\calP).
\end{eqnarray*}
\end{proof}

\begin{remark}\rm
For product weights $\bsgamma_{\setu}=\prod_{j \in \setu} \gamma_j$
with a sequence $(\gamma_j)_{j \ge 1}$ of positive weights we have
$$
\widetilde{\gamma}_{\setu}=\left(\prod_{j \in \setu}
\frac{2 \gamma_j}{2^{1/p^\ast}}\right)  \prod_{j \in [d]\setminus\setu}
\left(1+ \frac{2^{p^\ast -1}}{p^*+1}\, \gamma_j^{p^*} 
\right)^{1/p^*} \quad\mbox{for}\quad p^*<\infty,
$$
and 
$$
\widetilde{\gamma}_{\setu}= \prod_{j \in \setu} (2 \gamma_j)
\quad\mbox{for}\quad p^*=\infty.
$$
\end{remark}

\begin{remark}\rm
To have small worst case error, one should use node sets with low
weighted $L_{p^*}$-discrepancy. These discrepancies 
have been well studied with
respect to both, the order of magnitude in $n$ as well as their dependence
on the dimension $d$. There are constructions of $n$-element point sets in
$[0,1)^d$ yielding a convergence rate of order $O((\log n)^{(d-1)/2}/n)$
if $p^* \in [1,\infty)$ and $O((\log n)^{d-1}/n)$ if $p^*=\infty$. 
Furthermore, conditions on the weights are known which guarantee
various kinds of tractability for the weighted discrepancy and hence for the 
corresponding integration problem in the ANOVA space. For information
see, for example, \cite{DP2010,DP2015,LP2003,NW2010} and the references therein.
\end{remark}

\subsection{A general lower bound for the worst case error} 
We now provide  the following general lower bound.

\begin{theorem}\label{thm:lbddimgen}
Assume that $\gamma_{\setu}>0$ for every $\emptyset \not=\setu \subseteq[d]$.
Then for every $p \in [1,\infty)$ there exists a positive
constant $c=c(p^\ast,d,\bsgamma)$ such that any QMC rule
based on an $n$-element point set in $[0,1)^d$ has the worst case error
bounded from below by
\begin{equation}\label{lbd:errDd}
{\rm error}(\QMC_{d,n};\calF_d) \ge c\, 
\frac{(\ln n)^{\frac{d-1}{2}}}{n}.
\end{equation}
For $p=\infty$ we have 
\begin{equation}\label{lbd:errDdpstern1}
{\rm error}(\QMC_{d,n};\calF_d) \ge c\, 
\frac{(\ln n)^{\frac{1}{2}}}{n}.
\end{equation}
For $d=2$ and $p=1$ the lower bound \eqref{lbd:errDd} can be improved to 
\begin{equation}\label{lbd:errD2infty}
{\rm error}(\QMC_{2,n};\calF_2) \ge c\, 
\frac{\ln n}{n}.
\end{equation}

\end{theorem}

For the proof we need the following technical lemma.

\begin{lemma}\label{lem:T_ell}
 Let $d\in\NN$, $d\ge 2$, and let $T_d (\ell)$ for $\ell\in [d]$ be defined by 
 \[
  T_d (\ell):= \frac{(\ell !)^2}{(d !)^2}.
 \]
Then it is true for every $\ell\in\{2,\ldots,d\}$ that 
\[
 T_d (\ell) > \sum_{k=1}^{\ell-1} {\ell \choose k} T_d (k).
\]
\end{lemma}
\begin{proof}
Let, for $\ell \in [d]$,
 \[
  s(\ell):= \sum_{k=1}^{\ell-1} {\ell \choose k} (k!)^2.
 \]
Showing the desired inequality is equivalent to showing that
$s(\ell) < (\ell !)^2$ for all $\ell \in\{2,\ldots,d\}$.
This is done by induction on $\ell$. It is easily checked that the assertion
holds for $\ell=1,2$.
Assume that we have $s(\ell) < (\ell!)^2$. Now we consider $s(\ell+1)$.
We have
\begin{eqnarray*}
s(\ell+1) & = & \sum_{k=1}^{\ell} {\ell+1 \choose k} (k!)^2 \\
& = & \sum_{k=1}^{\ell}\frac{\ell +1}{k} {\ell \choose k-1} (k!)^2 \\
& = & \sum_{k=0}^{\ell-1} \frac{\ell +1}{k+1} {\ell \choose k} ((k+1)!)^2\\
& = & \ell+1 + \sum_{k=1}^{\ell -1} {\ell \choose k} (k!)^2 (\ell+1) (k+1)\\
& \le & \ell +1 + (\ell+1)\, \ell\, s(\ell)\\
& < &  \ell +1 + (\ell+1)\, \ell\, (\ell!)^2\\
& \le & (\ell+1)^2 (\ell!)^2\\
& = & ((\ell+1)!)^2.
\end{eqnarray*}
This completes the proof.
\end{proof}

\begin{proof}[Proof of Theorem~\ref{thm:lbddimgen}]
Note that for every $p^\ast \in (1,\infty]$ there exists a $C=C(d,p^\ast)>0$
such that for every $n$-element point set $\calP$ in $[0,1)^d$ we
have 
\begin{equation}\label{lbd:Lpdisc}
L_{p^\ast}(\calP) \ge C(d,p^\ast) \, \frac{(\ln n)^{\frac{d-1}{2}}}{n}.
\end{equation}
For $p^* \ge 2$ this is a famous result by Roth~\cite{roth1954} that was extended later by Schmid~\cite{schm1977} to the case $p^* \in (1,2)$. For $p^*=1$ we always have
\begin{equation}\label{lbd:Lpdisc1}
L_{1}(\calP) \ge C(d,p^\ast) \, \frac{(\ln n)^{\frac{1}{2}}}{n},
\end{equation}
as shown by Hal\'{a}sz~\cite{hal1981} for $d=2$, but the result holds for all $d \in \NN$ (cf.~\cite{ABL2014}). 
For $d=2$ and  $p^*=\infty$ the lower bound \eqref{lbd:Lpdisc} can, according to Schmidt~\cite{schm1972}, be tightened to
\begin{equation}\label{lbd:Linftydisc}
L_{\infty}(\calP) \ge C(2,\infty) \, \frac{\ln n}{n}.
\end{equation}

We now use the sequence $T_d(\ell)$ for $\ell\in [d]$
from  Lemma~\ref{lem:T_ell}, 
$T_d (\ell):= (\ell !)^2 / (d !)^2$.

According to \eqref{prcor1:eq1}, we have for $p>1$, and hence $p^* < \infty$,  
\begin{equation}\label{eq:lbd_formula}
{\rm error}(\QMC_{d,n};\calF_d)  = 
\left[\sum_{\emptyset\not=\setu \subseteq [d]}\gamma_\setu^{p^*}\,
\int_{[0,1]^{|\setu|}}\left| 
\sum_{k=1}^{|\setu|} (-1)^k \sum_{\setv \subseteq \setu\atop |\setv|=k} \Delta_{\calP_{\setv}}(\bst_{\setv})  
\prod_{i \in\setu\setminus \setv} t_i \right|^{p^*}\rd\bst_\setu\right]^{1/p^*}.
\end{equation}

Assume that $p < \infty$ and hence $p^*>1$. Suppose first that 
\begin{equation}\label{eq:lbd_dim1}
 L_ {p^{\ast}}(\calP_{\{i\}}) \ge T_d(1) C(d,p^\ast) \, \frac{(\ln n)^{\frac{d-1}{2}}}{n}
\end{equation}
for some $i\in [d]$. Then we obtain from \eqref{prcor1:eq1} or \eqref{eq:lbd_formula}, respectively, that 
\[
{\rm error}(\QMC_{d,n};\calF_d) \ge \gamma_{\{i\}}L_ {p^{\ast}}(\calP_{\{i\}}) \ge \gamma_{\{i\}} 
T_d(1) C(d,p^\ast) \, \frac{(\ln n)^{\frac{d-1}{2}}}{n}.
\]

If \eqref{eq:lbd_dim1} does not hold for any $i\in[d]$, let $\ell\in \{2,\ldots,d\}$ be minimal such that the following two 
conditions hold:
\begin{itemize}
 \item[(i)] There exists $\setu \subseteq [d]$, $\setu\neq \emptyset$, with $\abs{\setu}=\ell$ such that
 \[
  L_{p^{\ast}} (\calP_{\setu}) \ge T_d(\ell) C(d,p^\ast) \, \frac{(\ln n)^{\frac{d-1}{2}}}{n},
 \]
 \item[(ii)] but
 \[
  L_{p^{\ast}} (\calP_{\setv}) < T_d(\abs{\setv}) C(d,p^\ast) \, \frac{(\ln n)^{\frac{d-1}{2}}}{n}
 \]
 for all $\setv \subsetneq \setu$, $\setv\neq \emptyset$.
\end{itemize}
Note that $\ell\ge 2$ since we assumed that \eqref{eq:lbd_dim1} does not hold for any $i\in[d]$, and $\ell\le d$ due 
to the fact that \eqref{lbd:Lpdisc} holds and $T_d (d)=1$. 

Let now $\ell$ be defined as above, and let $\setu \subseteq [d]$ be such that Condition (i) holds. Then it follows from 
\eqref{eq:lbd_formula} that
\begin{eqnarray*}
{\rm error}(\QMC_{d,n};\calF_d) &\ge& 
\left[\gamma_\setu^{p^*}\,
\int_{[0,1]^{\ell}}  \abs{ (-1)^{\ell} \Delta_{\calP_{\setu}} (\bst_{\setu}) + 
 \sum_{k=1}^{\ell -1 } (-1)^k \sum_{\setv \subseteq \setu\atop |\setv|=k} \Delta_{\calP_{\setv}}(\bst_{\setv})  
 \prod_{i \in\setu\setminus \setv} t_i }^{p^*}\rd\bst_{\setu}\right]^{1/p^*}\\
 &=& 
\left[\gamma_\setu^{p^*}\,
\int_{[0,1]^{\ell}}  \abs{  \Delta_{\calP_{\setu}} (\bst_{\setu}) - (-1)^{\ell-1} 
 \sum_{k=1}^{\ell -1 } (-1)^k \sum_{\setv \subseteq \setu\atop |\setv|=k} \Delta_{\calP_{\setv}}(\bst_{\setv})  
 \prod_{i \in\setu\setminus \setv} t_i }^{p^*}\rd\bst_{\setu}\right]^{1/p^*}.
\end{eqnarray*}
Let now 
\[
 g_{\setu} (\bst_{\setu}):= (-1)^{\ell-1} 
 \sum_{k=1}^{\ell -1 } (-1)^k \sum_{\setv \subseteq \setu\atop |\setv|=k} \Delta_{\calP_{\setv}}(\bst_{\setv})  
 \prod_{i \in\setu\setminus \setv} t_i 
 \quad\quad \mbox{for}\quad \bst_{\setu}\in [0,1]^{\ell}.
\]
Then,
\begin{eqnarray*}
 {\rm error}(\QMC_{d,n};\calF_d) & \ge &  \gamma_{\setu} 
 \left[ \int_{[0,1]^d} \left|\Delta_{\calP_{\setu}}(\bst_{\setu})-g_{\setu}(\bst_{\setu})\right|^{p^\ast} \rd \bst_{\setu}\right]^{1/p^\ast}\\
 & = & \gamma_{\setu} \|\Delta_{\calP_{\setu}} - g_{\setu}\|_{L_{p^\ast}}\\
 & \ge & \gamma_{\setu} \left|\|\Delta_{\calP_{\setu}}\|_{L_{p^\ast}} - \|g_{\setu}\|_{L_{p^\ast}} \right|\\
 & = & \gamma_{\setu} \left|L_{p^\ast}(\calP_{\setu}) - \|g_{\setu}\|_{L_{p^\ast}} \right|.
\end{eqnarray*}

However,
\begin{eqnarray*}
 \|g_{\setu}\|_{L_{p^\ast}} &\le &\sum_{k=1}^{\ell-1} \sum_{\setv \subseteq \setu\atop |\setv|=k}  
 \norm{\left(\prod_{i \in\setu\setminus \setv} t_i\right)  \Delta_{\calP_{\setv}}(\bst_{\setv})}_{L_{p^\ast}}\\
 &\le& \sum_{k=1}^{\ell-1} \sum_{\setv \subseteq \setu\atop |\setv|=k}
  \left(\frac{1}{(p^\ast +1)^{1/p^\ast}}\right)^{\ell - k} \norm{\Delta_{\calP_{\setv}}(\bst_{\setv})}_{L_{p^\ast}}\\
 &\le& \sum_{k=1}^{\ell-1} \sum_{\setv \subseteq \setu\atop |\setv|=k} 
 L_{p^{\ast}} (\calP_{\setv})\\
 &\le& \sum_{k=1}^{\ell-1} {\ell \choose k} T_d(k)\, C(d,p^\ast) \, \frac{(\ln n)^{\frac{d-1}{2}}}{n}\\
 &=&C(d,p^\ast) \, \frac{(\ln n)^{\frac{d-1}{2}}}{n}\sum_{k=1}^{\ell-1} {\ell \choose k} T_d(k).
\end{eqnarray*}
This yields 
\begin{eqnarray*}
 {\rm error}(\QMC_{d,n};\calF_d) 
 & \ge & \gamma_{\setu} \left|L_{p^\ast}(\calP_{\setu}) - \|g_{\setu}\|_{L_{p^\ast}} \right|\\
 & \ge &C(d,p^\ast) \, \frac{(\ln n)^{\frac{d-1}{2}}}{n} \left(T_d (\ell) - \sum_{k=1}^{\ell-1} {\ell \choose k} T_d(k)\right)\\
 & = &C(d,p^\ast) \, \frac{(\ln n)^{\frac{d-1}{2}}}{n} c_{\ell},
\end{eqnarray*}
where $c_{\ell}>0$ by Lemma~\ref{lem:T_ell}.
The same proof idea with some obvious modifications also works for $p=1$
and hence $p^*=\infty$. Therefore \eqref{lbd:errDd} is completely proven.
Furthermore, the same proof but with all terms $(\ln n)^{\frac{d-1}{2}}$
replaced by $(\ln n)^{\frac{1}{2}}$ works for \eqref{lbd:errDdpstern1}
($p=\infty$ and $p^*=1$).

It remains to prove \eqref{lbd:errD2infty}. For $d=2$ and $p=1$
(i.e., $p^*=\infty$)  we have 
\begin{eqnarray}\label{lbdd2:eq1infty}
\lefteqn{{\rm error}(\QMC_{2,n};\calF_2)}\nonumber\\
  & = & \max\left\{ \gamma_{\{1\}} L_{\infty}(\calP_{\{1\}}),
\gamma_{\{2\}} L_{\infty}(\calP_{\{2\}}),\gamma_{\{1,2\}}
\sup_{(t_1,t_2) \in [0,1]^2} |\Delta_{\calP}(t_1,t_2) - g(t_1,t_2)|\right\},
\end{eqnarray}
where $$g(t_1,t_2):=t_1 \Delta_{\calP_{\{2\}}}(t_2)+ t_2 \Delta_{\calP_{\{1\}}}(t_1).$$

Let $C=C(2,\infty)$ from \eqref{lbd:Linftydisc}. Now we consider two cases:
\begin{description}
\item[Case 1:] $L_ {\infty}(\calP_{\{i\}}) \ge \tfrac{C}{4} \tfrac{\ln n}{n}$ for at least one $i \in [2]$, say for $i=1$. Then we obtain from \eqref{lbdd2:eq1infty} that
\[{\rm error}(\QMC_{2,n};\calF_2) \ge \gamma_{\{1\}}L_ {\infty}(\calP_{\{1\}}) \ge \gamma_{\{1\}} \frac{C}{4} \frac{\ln n}{n}.
\]
 \item[Case 2:] $L_ {p^{\ast}}(\calP_{\{i\}}) < \tfrac{C}{4}  \tfrac{\ln n}{n}$ for all $i\in [2]$. Then  we obtain from \eqref{lbdd2:eq1infty} that 
 \begin{eqnarray*}
{\rm error}(\QMC_{2,n};\calF_2) & \ge &  \gamma_{\{1,2\}} \sup_{(t_1,t_2)\in [0,1]^2} \left|\Delta_{\calP}(t_1,t_2)-g(t_1,t_2)\right|\\
& = & \gamma_{\{1,2\}} \|\Delta_{\calP} - g\|_{L_{\infty}}\\
& \ge & \gamma_{\{1,2\}} \left|\|\Delta_{\calP}\|_{L_{\infty}} - \|g\|_{L_{\infty}} \right|\\
& = & \gamma_{\{1,2\}} \left|L_{\infty}(\calP) - \|g\|_{L_{\infty}} \right|.
\end{eqnarray*}
We have 
\begin{eqnarray*}
 \|g\|_{L_{\infty}} & \le & \|t_1 \Delta_{\calP_{\{2\}}}(t_2)\|_{L_{\infty}} + \|t_2 \Delta_{\calP_{\{1\}}}(t_1)\|_{L_{\infty}}\\
 & = &  \left(L_{\infty}(\calP_{\{1\}})+L_{\infty}(\calP_{\{2\}})\right)<\frac{C}{2} \frac{\ln n}{n}.
\end{eqnarray*}
Hence, together with \eqref{lbd:Linftydisc}, we get
\[{\rm error}(\QMC_{2,n};\calF_2) >  \gamma_{\{1,2\}}
\frac{C}{2} \frac{\ln n}{n}.
\]
\end{description}
In any case we have
\[{\rm error}(\QMC_{2,n};\calF_2) >  c(p^\ast,\bsgamma) \,
\frac{\ln n}{n},
\]
where we can choose
\[c(p^\ast,\bsgamma)=C\,\min_{\setu \not=\emptyset}
\frac{\gamma_{\setu}}{2^{|\setu|}},
\]
with $C$ taken from \eqref{lbd:Linftydisc}.
\end{proof}

\section{Low-dimensional cases}\label{sec:lowD}

In this section we study the special low-dimensional
cases $d \in \{1,2\}$ in greater detail and provide optimal QMC rules.

\subsection{The case 1D}\label{sec:1D}

From \eqref{prcor1:eq1} we have: 

\begin{corollary} For any $n$-element point set $\calP$ in
$[0,1)$, the error of the corresponding QMC method is 
(modulo $\gamma_{\{1\}}$) the $L_{p^*}$-discrepancy of $\calP$,
\[  {\rm error}(\QMC_{1,n},\calF_1)=\gamma_{\{1\}}\,L_{p^*}(\calP).
\]
\end{corollary}

Next we show that the composite midpoint rule 
\[\QMC^{\rm MP}_{1,n}(f)=\frac1n\,\sum_{j=1}^nf(x_j)\quad\mbox{with}\quad
x_j=\frac{2j-1}{2\,n}
\]
is optimal among all $\QMC$ rules based on $n$ nodes in $[0,1]$.
This is equivalent to the fact that the point set formed by 
\begin{equation}\label{def:xMPR}
y_j=\frac{2j-1}{2n}\ \ \mbox{for $j=1,\ldots,n$}
\end{equation}
has optimal $L_{p^\ast}$-discrepancy among all $n$-element point
sets in $[0,1)$. The latter is well known for
$p^\ast \in \{2,\infty\}$, see \cite{nie1973}, 
and has been shown recently in \cite{KPa19} for arbitrary
$p^\ast \in [1,\infty]$. Here we 
give an elementary proof
for the case of general $p^\ast \in [1,\infty]$.

We begin with the following lemma.

\begin{lemma}\label{lem:simp}
For any $p^* \in [1,\infty]$, if an $n$-element point set
$\calP=\{x_1,\dots,x_n\}$ has the least 
$L_{p^*}$-discrepancy, then each subinterval $[\tfrac{j-1}n,\tfrac{j}n)$, $j =1,\ldots, n$,
contains exactly one point from $\calP$.
\end{lemma}
\begin{proof} We provide a proof for $p^*<\infty$ only. The case
for $p^*=\infty$ is addressed at the end of the proof of the next theorem. 

Let $\calP=\{x_1,\dots,x_n\}$ with $0\le x_1\le \cdots\le x_n\le1$. We have 
\[\left[L_{p^*}(\calP)\right]^{p^*}\,=\,\sum_{j=1}^n e_j(\calP),
\quad\mbox{where}\quad e_j(\calP)\,=\,\int_{(j-1)/n}^{j/n}
|\Delta_{\calP}(t)|^{p^*}\rd t.
\]
To simplify the notation in this proof, we introduce
\[
  s_j\,=s_j(\calP)\,:=\,|\calP\cap[0,\tfrac{j-1}n)|\quad\mbox{for}
  \quad j\,=\,1,\dots,n. 
\]
Then the  $e_j(\calP)$'s satisfy the following properties:
\[e_j(\calP)\,=\,\int_{(j-1)/n}^{j/n}\left|\frac{s_j}n-t\right|^{p^*}\rd t
\]
if $\calP\cap\left[\tfrac{j-1}n,\tfrac{j}n\right)\,=\emptyset$,
and
\[e_j(\calP)\,=\,\int_{(j-1)/n}^{x_{j,1}}
\left|\frac{s_j}n-t\right|^{p^*}\rd t
+\int_{x_{j,1}}^{x_{j,2}}\left|\frac{s_j+1}n-t\right|^{p^*}\rd t+\cdots+
\int_{x_{j,k_j}}^{j/n}\left|\frac{s_j+k_j}n-t\right|^{p^*}\rd t,
\]
if $\calP\cap\left[\tfrac{j-1}n,\tfrac{j}n\right)=\{x_{j,1},\dots,x_{j,k_j}\}$.

Suppose, by a contradiction, that for some $\ell$, the subinterval 
$[(\ell-1)/n,\ell/n)$ does not contain any point from $\calP$.
If there are more such subintervals, then we choose $\ell$ to be
the smallest index. We now consider two cases.

{\bf Case 1:} Suppose that there is $m<\ell$ such that $[(m-1)/n,m/n)$
contains more than one point from $\calP$. Choose the largest such $m$
if there are more of such subintervals. Then
\[\left[\tfrac{m-1}n,\tfrac{m}n\right)\cap\calP\,=\,
  \{x_{m,1},\dots,x_{m,k_m}\}\quad\mbox{with}\quad
  k_m\ge2,
\]
\[\left[\tfrac{j-1}n,\tfrac{j}n\right)\cap\calP
  \,=\,\{x_{j,1}\}\quad\mbox{for}\quad
  j=m+1,\dots,\ell-1,\quad\mbox{and}\quad
  \left[\tfrac{\ell-1}n,\tfrac\ell{n}\right)\cap\calP\,=\,\emptyset. 
\]
Note that, due to $k_m\ge2$, we have 
\begin{equation}\label{s-prop}
s_{m+1}\,\ge\,m-1+k_m\,\ge\,m+1,\quad\mbox{and}\quad
  s_{m+j}\,=\,s_{m+1}+j-1\,\ge\,m+j\quad\mbox{for}\quad j\,=\,2,\dots,\ell-m.
\end{equation}

Consider next $\widetilde\calP$ which is obtained from
$\calP$ by removing the point $x_{m,k_m}$ and adding $y_m$
inside $((\ell-1)/n,\ell/n]$. Clearly $e_i(\calP)=e_i(\widetilde\calP)$
for $i<m$ and $i>\ell$. Note that
\[s_{m+j}(\widetilde\calP)\,=\,s_{m+j}(\calP)-1\quad\mbox{for}\quad
  j\,=\,1,\dots,\ell-m.
\]
Therefore, 
\begin{eqnarray*}
|L_{p^*}(\calP)|^{p^*}-|L_{p^*}(\widetilde\calP)|^{p^*}&=&
\int_{x_{m,k_m}}^{x_{m+1,1}}\left[\left|\frac{s_{m+1}}n-t\right|^{p^*}
-\left|\frac{s_{m+1}-1}n-t\right|^{p^*}\right]\rd t\\
&&+\sum_{j=1}^{\ell-m-2}\int_{x_{m+j,1}}^{x_{m+j+1,1}}\left[
\left|\frac{s_{m+j+1}}n-t\right|^{p^*}-\left|\frac{s_{m+j+1}-1}n-t\right|^{p^*}
  \right]\rd t\\
&&+\int_{x_{\ell-1,1}}^{y_m}\left[\left|\frac{s_\ell}n-t\right|^{p^*}-
\left|\frac{s_\ell-1}n-t\right|^{p^*}\right]\rd t.
\end{eqnarray*}
Due to \eqref{s-prop}, all integrals in the sum above are positive.
Indeed, if $s_{m+j+1}\ge m+j+2$ then
$s_{m+j+1}-1\ge m+j+1> n  x_{m+j+1,1}$
and
\[  \frac{s_{m+j+1}}n-t\, >\, \frac{s_{m+j+1}-1}n-t>0 \quad
\mbox{for any}\quad t\in[x_{m+j,1},x_{m+j+1,1}].
\]
If
$s_{m+j+1}=m+j+1$ then $\tfrac{s_{m+j+1}-1}n=\tfrac{m+j}n\in
[x_{m+j,1},x_{m+j+1,1}]$ and, therefore
\[\int_{x_{m+j,1}}^{x_{m+j+1,1}}\left(\frac{m+j+1}n-t\right)^{p^*}\rd t
\,>\, \int_{x_{m+j,1}}^{x_{m+j+1,1}}\left|\frac{m+j}n-t\right|^{p^*}\rd t.
\]
Hence $L_{p^*}(\calP)>L_{p^*}(\widetilde\calP)$.

{\bf Case 2:} Suppose that Case 1 is not applicable.
Let $\ell$ be as before and let  $m> \ell$ be the smallest index such that
$[(m-1)/n,m/n)$ contains more than one point of $\calP$. Let
$\calP\cap[(m-1)/n,m/n)=\{x_{m,1},\dots,x_{m,k_m}\}$ with $k_m\ge2$.
Using similar arguments as in Case 1, one can verify that
the $L_{q^*}$-discrepancy of $\calP$ is larger than the discrepancy
of $\widetilde{\calP}$, where now $\widetilde\calP$ has $x_{m,1}$
replaced by $y_{\ell,1}\in[(\ell-1)/n,\ell/n)$. 
\end{proof}

\begin{theorem}\label{lem:midp}
We have 
\begin{equation}\label{err-MP}
  {\rm error}(\QMC_{1,n}^{\rm MP};\calF_1)\,=\,
\frac{\gamma_{\{1\}}}{2\,(p^*+1)^{1/p^*}\,n}
\end{equation}
with the convention that $(p^*+1)^{1/p^*}=1$ for $p^*=\infty$.
Moreover,
\begin{equation}\label{opt-MP}
{\rm error}(\QMC_{1,n}^{\rm MP};\calF_1)\,=\,
\min\{{\rm error}(\QMC_{1,k};\calF_1)\,:\,k\le n\}.
\end{equation}
\end{theorem}

\begin{proof}
We begin with $p^*<\infty$. Let $y_1,\ldots,y_n$ be the nodes of the midpoint rule given by \eqref{def:xMPR}. From \eqref{eq:4} we have
\begin{eqnarray*}
{\rm error}(\QMC_{1,n}^{\rm MP};\calF_1) &=& \gamma_{\{1\}}\,\bigg[
\int_0^1\left|\frac{|\{k\in\{1,\ldots,n\}\,:\,y_k<t\}|}n-t\right|^{p^*}
    \rd t\bigg]^{1/p^*}\\
  &=&\gamma_{\{1\}}\bigg[\sum_{j=0}^n\int_{y_j}^{y_{j+1}}
   \left|\frac{|\{k\in\{1,\ldots,n\}\,:\,y_k<t\}|}n-t\right|^{p^*}\rd t\bigg]^{1/p^*},
\end{eqnarray*}
where we put $y_0=0$ and $y_{n+1}=1$. The first and last integrals in the sum
above are  equal to
\[\int_0^{y_1}|0-t|^{p^*}\rd t=\int_{y_n}^1|1-t|^{p^*}\rd t=
 \frac1{(p^*+1)\,(2\,n)^{p^*+1}}.
\]
The other integrals are equal to 
\[\int_{y_j}^{y_{j+1}}\left|\frac{j}n-t\right|^{p^*}\rd t=
2\,\int_{j/n}^{(2j+1)/(2n)}\left(t-\frac{j}n\right)^{p^*}\rd t
  = \frac{2}{(p^*+1)\,(2\,n)^{p^*+1}}.
\]
This proves that the sum of the integrals is equal to 
\[\sum_{j=0}^n\int_{y_j}^{y_{j+1}}
\left|\frac{|\{k\in\{1,\ldots,n\}\,:\,y_k<t\}|}n-t\right|^{p^*}\rd t
=\frac1{(p^*+1)\,(2\,n)^{p^*}},
\]
which completes the proof for $p^*<\infty$.

For $p^*=\infty$, we have
\begin{eqnarray*}
{\rm error}(\QMC_{1,n}^{\rm MP};\calF_1)
&=& \gamma_{\{1\}}\,
\sup_{t\in[0,1]}\left|\frac{|\{j\in\{1,\ldots,n\}\,:\,y_j<t\}|}n-t\right|\\
&=&\gamma_{\{1\}}\max_{j=0,\dots,n}\,\max_{t\in[y_j,y_{j+1}]}\left|
\frac{j}n-t\right|=\frac{\gamma_{\{1\}}}{2\,n}.
\end{eqnarray*}
This completes the proof of \eqref{err-MP}.

To show \eqref{opt-MP}, consider a general point set $\calP
=\{x_1,\dots,x_n\}$,
\[
0=:x_0\le x_1\le\cdots\le x_n\le x_{n+1}:=1.
\]
For $p^*<\infty$, the worst case error of the corresponding QMC rule $\QMC_{1,n}$
raised to the power $p^*$ is equal to 
\begin{equation}\label{eq-cos}
  E(x_1,\dots,x_n):=
  \gamma_{\{1\}}^{p^*}\,\sum_{\ell=0}^n\int_{x_\ell}^{x_{\ell+1}}
\left|\frac\ell{n}-t\right|^{p^*} \rd t.
\end{equation}
Its partial derivative with respect to $x_k$ is
\[
\gamma_{\{1\}}^{p^*}\,\left(\left|\frac{k-1}n-x_k\right|^{p^*}-
\left|\frac{k}n-x_k\right|^{p^*}\right),
\]
which is zero if and only if 
\[\frac{k-1}n-x_k=-\left(\frac{k}n-x_k\right),\quad\mbox{i.e., iff }\ 
x_k=\frac{2k-1}{2n}. 
\]
This means that the only possible extremal point is given by
$x_k=\frac{2k-1}{2n}$ for $1\le k\le n$.
It is easy to see that the minimum of $E$ is not attained if $x_1=0$ and/or
$x_n=1$. Due to Lemma \ref{lem:simp} it is not attained if $x_i=x_{i+1}$
for some $i$, which completes the proof of \eqref{opt-MP} for $p^*<\infty$.

For $p^*=\infty$, we need to show
\begin{equation}\label{todo}
  \gamma_{\{1\}}\,\max_{\ell=0,\dots,n}\ \max_{t\in[x_\ell,x_{\ell+1}]}
  \left|\frac{\ell}n-t\right|\,\ge\,\frac{\gamma_{\{1\}}}{2\,n}.
\end{equation}
To prove \eqref{todo}, suppose by contrary that for some point set
$\calP=\{x_1,\dots,x_n\}$ it holds that
\[\max_{\ell=0,\dots,n}\ \max_{t\in[x_\ell,x_{\ell+1}]}\left|\frac\ell{n}-t\right|
  \,<\,\frac1{2\,n}.
\]
Then $|0-x_1|<1/(2n)$ and $|1-x_n|<1/(2n)$.
Moreover, 
\[\left|\frac\ell{n}-x_\ell\right|\,<\,\frac1{2\,n}\quad\mbox{and}\quad
  \left|x_{\ell+1}-\frac\ell{n}\right|\,<\,\frac1{2\,n}, 
\]
which implies that
\[x_{\ell+1}-x_\ell\,\le\, |\ell/n-x_\ell|+|x_{\ell+1}-\ell/n|\,<\,1/n
  \quad\mbox{for\ }\ell\,=1,\dots,n-1.
\]
Therefore
\[ 1\,=\,x_1+(x_2-x_1)+\cdots+(x_n-x_{n-1})+1-x_n\,<\,
  \frac1{2\,n}+\frac{n-1}{n}+\frac1{2\,n}\,=\,1,
\]
which is a contradiction. 

This completes the proof of \eqref{todo} and of the theorem.
\end{proof}

\begin{remark}\label{rem10}\rm
From \cite[Theorem~8]{KPW2016} it follows that for $p=2$ and
$\gamma_\emptyset=1$ the norms of the corresponding embeddings are equal and 
\[\|\imath\|\,=\,\|\imath^{-1}\|\,=\,\left(1+\frac{\gamma_{\{1\}}}{\sqrt{3}}\,
\left(\sqrt{1+\frac{\gamma_{\{1\}}^2}{12}}+
\frac{\gamma_{\{1\}}}{\sqrt{12}}\right)\right)^{1/2},
\]
which could be large. For instance for $\gamma_{\{1\}}=1,2,3$, these norms 
are equal to $1.329\dots,1.732\dots,$ and $2.188\dots$, respectively. 
Hence, using the embedding approach we would get 
$\|\imath\| \gamma_{\{1\}}/(2\sqrt{3} n)$
as an upper bound for the error of $\QMC^{\rm MP}_{1,n}$.
Proceeding directly as in  Theorem~\ref{lem:midp}, however, we get the
exact value of the error of the midpoint rule which is  
$\gamma_{\{1\}}/(2\sqrt{3} n)$.
\end{remark}

\subsection{The case 2D}

Now we consider the two-dimensional case and show that here the lower bound in Theorem~\ref{thm:lbddimgen} is best 
possible with respect to the order of magnitude in $n$. In the following we assume that the two-dimensional point sets $\calP=\{(x_j,y_j)\ : \ j=1,\ldots,n\}$ under consideration are projection regular in the sense that
\begin{equation}\label{projreg}
\{x_j \ : \ j=1,\ldots,n\}=\{y_j \ : \ j=1,\ldots,n\}=\{j/n\ : \ j=0,\ldots,n -1\}.  
\end{equation}

We will prove the following result:

\begin{theorem}\label{thm2}
Let $\QMC_{2,n}$ be the
QMC rule based on a two-dimensional point set $\calP$ that satisfies projection regularity \eqref{projreg}. Then we have
\begin{eqnarray*}
{\rm error}(\QMC_{2,n};\calF_2) & \le & \frac{1}{n}
  \left[\frac{\gamma_{\{1\}}^{p^*}+\gamma_{\{2\}}^{p^*}}{p^\ast +1}  + 3^{p^\ast -1} \gamma_{\{1,2\}}^{p^*} \left(\frac{2}{(p^\ast +1)^2} + (n L_{p^\ast}(\calP))^{p^\ast}\right) \right]^{1/p^*}. 
\end{eqnarray*}
On  the other hand, there exists a positive number
$C=C(p^*)$ such that  
\begin{eqnarray*}
  {\rm error}(\QMC_{2,n};\calF_2) & \ge & C \,
  \gamma_{\{1,2\}} \, L_{p^\ast}(\calP).
\end{eqnarray*}  
\end{theorem}

For the proof we need the following easy lemma:

\begin{lemma}\label{le1}
If $\calP=\{(x_j,y_j)\ : \ j=1,\ldots,n\}$ satisfies \eqref{projreg},
we have for $t \in [0,1]$ that
\[\sum_{j=1}^n 1_{[0,t)}(x_j)= \sum_{j=1}^n 1_{[0,t)}(y_j)
=\lceil n t\rceil.
\]
\end{lemma}

\begin{proof}
Since $\calP$ satisfies \eqref{projreg} we obtain   
$$ \sum_{j=1}^n 1_{[0,t)}(x_j) = \sum_{j=0\atop j < nt}^{n-1} 1= \sum_{j=0}^{\lceil n t\rceil -1} 1=\lceil n t\rceil.$$
\end{proof}

\begin{proof}[Proof of Theorem~\ref{thm2}]
First we show the upper bound. We need to study 
\begin{equation}\label{tostudyd2}
\sum_{j=1}^n \prod_{i \in \setu}(t_i-1_{[0,t_i)}(x_{j,i}))
\end{equation}
for $\setu=\{1\}$, $\setu=\{2\}$, and $\setu=\{1,2\}$,
where $x_{j,1}=x_j$ and $x_{j,2}=y_j$.
\begin{itemize}
\item $\setu=\{1\}$: According to Lemma~\ref{le1},
Eq.~\eqref{tostudyd2} is
\[\sum_{j=1}^n (t_1-1_{[0,t_1)}(x_j))=n t_1 - \lceil n t_1\rceil
\]
and hence
\[\left|\sum_{j=1}^n (t_1-1_{[0,t_1)}(x_j))\right|= \lceil n
t_1\rceil -n t_1.\]
This implies that
\begin{eqnarray*}
  \int_0^1 \left|\sum_{j=1}^n (t_1-1_{[0,t_1)}(x_j))\right|^{p^\ast} \rd t_1 & = & \int_0^1 (\lceil n t_1\rceil -n t_1)^{p^\ast} \rd t_1\\
  & = & \sum_{k=1}^{n} \int_{\frac{k-1}{n}}^{\frac{k}{n}} (k-n t_1)^{p^\ast} \rd t_1\\
& = & \sum_{k=1}^n \frac{1}{n} \int_0^1 y^{p^\ast} \rd y\\
  & = & \frac{1}{p^\ast +1},
 \end{eqnarray*}
where we used the substitution $y=k-n t_1$.

\item $\setu=\{2\}$: In this case, according to Lemma~\ref{le1},
 Eq.~\eqref{tostudyd2} is
\[\sum_{j=1}^n (t_2-1_{[0,t_2)}(y_j))=n t_2 - \lceil n t_2\rceil\]
and hence, as above,
\[\int_0^1 \left|\sum_{j=1}^n (t_2-1_{[0,t_2)}(y_j))\right|^{p^\ast}\rd t_2
  =  \frac{1}{p^\ast +1}.
\]

\item $\setu=\{1,2\}$: Here, again according to Lemma~\ref{le1}, Eq.~\eqref{tostudyd2} is
 \begin{eqnarray*}
 \lefteqn{\sum_{j=1}^n(t_1-1_{[0,t_1)}(x_j)) (t_2-1_{[0,t_2)}(y_j))}\\
  & = & n t_1 t_2 -t_2 \lceil n t_1\rceil - t_1 \lceil n t_2 \rceil + \sum_{j=1}^n 1_{[0,t_1)}(x_j) 1_{[0,t_2)}(y_j)\\
  & = & 2 n t_1 t_2 -t_2 \lceil n t_1\rceil - t_1 \lceil n t_2 \rceil + \left(\sum_{j=1}^n 1_{[0,t_1)}(x_j) 1_{[0,t_2)}(y_j) - n t_1 t_2\right)\\
  & = & 2 n t_1 t_2 -t_2 \lceil n t_1\rceil - t_1 \lceil n t_2 \rceil + n \Delta_{\calP}(t_1,t_2).
 \end{eqnarray*}

Taking the absolute value and the $p^\ast$-th power we obtain with Lemma~\ref{le:easy} that
 \begin{eqnarray*}
 \lefteqn{\left|\sum_{j=1}^n(t_1-1_{[0,t_1)}(x_j)) (t_2-1_{[0,t_2)}(y_j))\right|^{p^\ast}}\\
  & \le & 3^{p^\ast-1} \left(t_1^{p^\ast}( \lceil n t_2\rceil- n t_2)^{p^\ast} + t_2^{p^\ast}( \lceil n t_1\rceil- n t_1)^{p^\ast} + (n |\Delta_{\calP}(t_1,t_2)|)^{p^\ast}\right).
\end{eqnarray*}
Now we integrate with respect to $(t_1,t_2) \in [0,1]^2$ and obtain
\begin{eqnarray*}
\int_{[0,1]^2}  \left|\sum_{j=1}^n(t_1-1_{[0,t_1)}(x_j)) (t_2-1_{[0,t_2)}(y_j))\right|^{p^\ast} \rd (t_1,t_2)
  & \le & 3^{p^\ast -1} \left(\frac{2}{(p^\ast +1)^2} + (n L_{p^\ast}(\calP))^{p^\ast}\right).
\end{eqnarray*}
\end{itemize}
 
From the error formula in Theorem~\ref{thm1} we obtain
\begin{eqnarray*}
{\rm error}(\QMC_{2,n};\calF_2) \le  \frac{1}{n}
  \left[\frac{\gamma_{\{1\}}^{p^*}+\gamma_{\{2\}}^{p^*}}{p^\ast +1} + 3^{p^\ast -1} \gamma_{\{1,2\}}^{p^*} \left(\frac{2}{(p^\ast +1)^2} + (n L_{p^\ast}(\calP))^{p^\ast}\right) \right]^{1/p^*}. 
\end{eqnarray*}
This proves the upper bound. 

Now we turn our attention to the lower bound. From the proof of Theorem~\ref{thm:lbddimgen} we know that  
\begin{eqnarray*}
\sum_{j=1}^n(t_1-1_{[0,t_1)}(x_j)) (t_2-1_{[0,t_2)}(y_j)) =  n \Delta_{\calP}(t_1,t_2)- g(t_1,t_2),
\end{eqnarray*}
where $$g(t_1,t_2):= -2 n t_1 t_2 +t_2 \lceil n t_1\rceil + t_1 \lceil n t_2 \rceil.$$

This yields, using again Theorem \ref{thm1},
\begin{eqnarray*}
{\rm error}(\QMC_{2,n};\calF_2) & \ge & \frac{1}{n}
  \gamma_{\{1,2\}}\,\left[
  \int_{[0,1]^{2}}\bigg|n \Delta_{\calP}(t_1,t_2)- g(t_1,t_2)
      \bigg|^{p^*}\rd(t_1,t_2)\right]^{1/p^*}\\
      & = & \frac{1}{n} \gamma_{\{1,2\}} \left\| n \Delta_{\calP}-g\right\|_{L_{p^\ast}}\\
      & \ge & \frac{1}{n} \gamma_{\{1,2\}} \left| \left\|n \Delta_{\calP}\right\|_{L_{p^\ast}} - \left\|g\right\|_{L_{p^\ast}}\right|\\
      & = & \frac{1}{n} \gamma_{\{1,2\}} \left| n L_{p^\ast}(\calP) - \left\|g\right\|_{L_{p^\ast}}\right|\\
      & = & \gamma_{\{1,2\}} \left|L_{p^\ast}(\calP) - \frac{1}{n}\left\|g\right\|_{L_{p^\ast}} \right|.
\end{eqnarray*}      
With the same methods as in the proof of the upper bound we can show that $$\left\|g\right\|_{L_{p^\ast}} \le \frac{2^{p^*+1}}{(p^* +1)^2}.$$ On the other hand, we know from \eqref{lbd:Lpdisc} that there exists an absolute constant $C>0$ such that $$L_{p^\ast}(\calP) \ge C \frac{\sqrt{\ln n}}{n}.$$ Hence we have 
\begin{eqnarray*}
{\rm error}(\QMC_{2,n};\calF_2) & \ge &   \gamma_{\{1,2\}}\left(L_{p^\ast}(\calP) - \frac{2^{p^*+1}}{(p^* +1)^2} \frac{1}{n}\right)\\
& \ge & \gamma_{\{1,2\}}L_{p^\ast}(\calP) \left(1 - \frac{2^{p^*+1}}{(p^* +1)^2} \frac{1}{C \cdot \sqrt{\ln n}}\right).
\end{eqnarray*}
Now, for $n$ large enough we have $$1 - \frac{2^{p^*+1}}{(p^* +1)^2} \frac{1}{C \cdot \sqrt{\ln n}}>0$$ and hence the result follows. 
\end{proof}

Several constructions of two-dimensional projection regular point sets with best possible order of $L_{p^\ast}$-discrepancy for all $p^\ast \in [1,\infty]$ are known, 
e.g., generalized Hammersley point sets \cite{FP2009}, shifted Hammersley point sets \cite{HKP2015,M2013} or digital NUT nets \cite{KP19}. 
As an example we would like to present the digitally shifted Hammersley point sets from \cite{HKP2015} in greater detail:

\begin{example}\label{ex:Hampts}\rm
Let $\bssigma=(\sigma_1,\sigma_2,\ldots,\sigma_{m}) \in \{0,1\}^m$. The two-dimensional digitally shifted Hammersley point set is given by
$$
 {\cal R}_{m,\bssigma} = \Big\{ \Big( \frac{t_m}{2}+\frac{t_{m-1}}{2^2}+\cdots + \frac{t_1}{2^m} ,  
                       \frac{t_{1} \oplus \sigma_1}{2}+\frac{t_{2}\oplus \sigma_2}{2^2}  +\cdots + \frac{t_m \oplus \sigma_m}{2^m} \Big) \ : \ 
                       t_1,\ldots,t_m \in \{0,1\}   \Big\}, 
$$
where $t \oplus \sigma =t+\sigma \pmod{2}$ for $t,\sigma \in \{0,1\}$. 
This point set contains $n=2^m$ elements. If $\bssigma=\bszero=(0,0,\ldots,0)$, we obtain the classical two-dimensional Hammersley point set. 
From \cite[Theorem~1]{HKP2015} we obtain that if $|\{j \ : \ \sigma_j=0\}|= \lfloor m/2\rfloor$, then for $p^\ast\in [1,\infty)$ 
we have 
$$
L_{p^\ast}({\cal R}_{m,\bssigma}) \asymp \frac{\sqrt{m}}{2^m} \asymp \frac{\sqrt{\ln n}}{n}.
$$ 
Furthermore, according to \cite{F2008,KLP2007}, 
$$
L_{\infty}({\cal R}_{m,\bssigma}) \asymp \frac{m}{2^m} \asymp \frac{\ln n}{n}.
$$
Since the point sets ${\cal R}_{m,\bssigma}$ are projection regular we obtain 
$$
{\rm error}(\QMC_{2,n};\calF_2) \asymp \left\{ 
\begin{array}{ll}
\frac{\sqrt{\ln n}}{n} & \mbox{if $p>1$,}\\[0.5em]
\frac{\ln n}{n} & \mbox{if $p=1$,}
\end{array}\right.
$$
where $n=2^m$, and these orders of magnitude are optimal according to Theorem~\ref{thm:lbddimgen}. We remark that for the classical two-dimensional Hammersley point set (where $\bssigma=\bszero$) we only get  $${\rm error}(\QMC_{2,n};\calF_2) \asymp \frac{\ln n}{n}$$ for all $p \in [1,\infty]$. This is optimal only for $p=1$ (i.e., $p^\ast=\infty$).
\end{example}

\section*{Acknowledgements}

The authors would like to thank Christoph Koutschan (RICAM, Austria) for helpful comments regarding the proof of Lemma \ref{lem:T_ell}.

\begin{small}
\noindent\textbf{Authors' addresses:}

\medskip
\noindent Peter Kritzer\\
Johann Radon Institute for Computational and Applied Mathematics (RICAM)\\
Austrian Academy of Sciences\\
Altenbergerstr.~69, 4040 Linz, Austria\\
E-mail: \texttt{peter.kritzer@oeaw.ac.at}

\medskip

\noindent Friedrich Pillichshammer\\
Institut f\"{u}r Finanzmathematik und Angewandte Zahlentheorie\\
Johannes Kepler Universit\"{a}t Linz\\
Altenbergerstr.~69, 4040 Linz, Austria\\
E-mail: \texttt{friedrich.pillichshammer@jku.at}

\medskip

\noindent G. W. Wasilkowski\\
Computer Science Department, University of Kentucky\\
301 David Marksbury Building\\
329 Rose Street\\
Lexington, KY 40506, USA\\
E-mail: \texttt{greg@cs.uky.edu}
\end{small}

\end{document}